\documentclass[11pt]{amsart}
\usepackage[twoside,left=3.1cm,right=3.1cm,bottom=2.66cm,top=2.66cm]{geometry}
\usepackage{amsfonts, enumerate}
\usepackage{amssymb}
\usepackage{amsmath}
\usepackage{latexsym}
\usepackage{mathrsfs}
\usepackage{amsthm}
\usepackage{verbatim}
\usepackage{amscd,hyperref}

\usepackage{comment}
\usepackage{framed}
\usepackage{color}
\definecolor{shadecolor}{gray}{0.875}

\newtheorem{thrm}{Theorem}[section]

\newtheorem{lem}[thrm]{Lemma}
\newtheorem{cor}[thrm]{Corollary}
\newtheorem{prop}[thrm]{Proposition}

\theoremstyle{definition}
\newtheorem{defn}[thrm]{Definition}

\newtheorem{exmple}[thrm]{Example}

\newtheorem{rmk}[thrm]{Remark}
\newtheorem{ques}[thrm]{Question}

\DeclareMathOperator{\Eff}{\overline{Eff}}

\DeclareMathOperator{\Nef}{Nef}

\DeclareMathOperator{\Hilb}{Hilb}

\DeclareMathOperator{\CH}{CH}

\DeclareMathOperator{\reldim}{reldim}

\def\O{\mathcal O}
\def\S{\mathcal S}
\def\ZZ{\mathbb Z}

\def\PP{{\mathbb P}}

\input xy
\xyoption{all}
\begin{document}

%
%\title{{\fontsize{15}{13}\selectfont Positivity of the diagonal}}
%\author{{\fontsize{12.5}{9}\selectfont Brian Lehmann and John Christian Ottem}}
%\date{}

\title{Positivity of the diagonal}
%\author{Brian Lehmann and John Christian Ottem}
%\date{}

\author{Brian Lehmann}
\address{Department of Mathematics \\
Boston College  \\
Chestnut Hill, MA \, \, 02467}
\email{lehmannb@bc.edu}

\author{John Christian Ottem}
\address{University of Oslo \\
Box 1053, Blindern, 0316 Oslo, Norway }
\email{johnco@math.uio.no}

\thispagestyle{empty}
%\vspace{-0.5cm}

\thispagestyle{empty}
\begin{abstract}
We study how the geometry of a projective variety $X$ is reflected in the positivity properties of the diagonal $\Delta_X$ considered as a cycle on $X\times X$. We analyze when the diagonal is big, when it is nef, and when it is rigid. In each case, we give several implications for the geometric properties of $X$. For example, when the cohomology class of $\Delta_X$ is big, we prove that the Hodge groups $H^{k,0}(X)$ vanish for $k>0$. We also classify varieties of low dimension where the diagonal is nef and big.
\end{abstract}
\maketitle
\vspace{-0.5cm}
\section{Introduction}
\thispagestyle{empty}

The geometry of a projective variety $X$ is determined by the positivity of the tangent bundle $T_{X}$. Motivated by the fact that $T_X$ is the normal bundle of the diagonal $\Delta_{X}$ in the self-product $X\times X$, we will in this paper study how the geometry of $X$ is reflected in the positivity properties of $\Delta_{X}$ itself, considered as a cycle on $X\times X$. The prototypical example of a variety with positive diagonal is projective space; the central theme of the paper is that positivity of the diagonal forces $X$ to be similar to projective space.
In dimension 1, this perspective is already quite vivid: $\PP^1$ is the only curve where the diagonal is an ample divisor; elliptic curves have nef, but not big diagonals; and for higher genus, the diagonal is contractible, hence `negative' in a very strong sense.

In general, when $X$ has dimension $n$, the diagonal determines a class in the space $N_n(X\times X)$ of $n$-dimensional cycles modulo numerical equivalence, and we are interested in how this class sits with respect to the various cones of positive cycles of $X\times X$.  Note that in the absence of the Hodge conjecture, we often do not even know the dimension of the space $N_n(X\times X)$.  Thus we develop techniques to prove positivity or rigidity without an explicit calculation of the positive cones. %We hope that our results can provide guiding examples and intuition for the theory of positivity for higher codimension cycles, which is still in its exploratory phase. %We hope that our results can provide guiding examples and intuition for the theory of positivity for higher codimension cycles, which is still in its exploratory phase.

The subsections below recall several different types of positivity and give a number of theorems illustrating each.  At the end of the introduction we will collect several examples of particular interest.

\subsection*{Big diagonal}

A cycle class $\alpha$ is said to be {\em big} if it lies in the interior of the closed cone generated by classes of effective cycles.  Bigness is perhaps the most natural notion of positivity for cycles.  We will also call a cycle {\em homologically big} if it is homologically equivalent to the sum of an effective $\mathbb{Q}$-cycle and a complete intersection of ample $\mathbb{Q}$-divisors.  Homological bigness implies bigness, and equivalence of the two notions would follow from the standard conjectures.   

The primary example of a variety with (homologically) big diagonal is projective space. In this case, the diagonal has a K\"unneth decomposition of the form
$$
\Delta_{X}=\sum_{p+q=n} \pi_1^*h^p\cdot \pi_2^*h^q
$$where $h$ is the hyperplane divisor, and this class is evidently big. Of course, the same argument applies also for the {\em fake projective spaces}, that is, smooth varieties $\neq \PP^n$ with the same betti numbers as $\PP^n$. In dimension 2, there are exactly 100 such surfaces \cite{py07}, \cite{cs10}, and they are all of general type. Thus unlike the case of curves, we now allow examples with positive Kodaira dimension, but which are still `similar' to projective space in the sense that they have the same Hodge diamond. 

More generally, homological bigness of the diagonal implies the vanishing of the `outer' Hodge groups of $X$.  Following ideas of \cite{LieFu}, we show:

\begin{thrm} \label{bignessthrm}
Let $X$ be a smooth projective variety.  If $\Delta_{X}$ is homologically big, then $H^{i,0}(X) = 0$ for $i >0$.
\end{thrm}

An interesting feature of this result is that the proof makes use of {\em non-algebraic} cohomology classes to control effective cycles. %In fact, the proof uses the Hodge--Riemann relations to construct interesting facets of the effective cone on $X \times X$ when $X$ carries non-vanishing forms.
When $X$ is a surface with a big diagonal, Theorem \ref{bignessthrm} implies the existence of a cohomological decomposition of the diagonal; we discuss this relationship in more depth in Section \ref{blochconjsec}.
%It turns out that Theorem \ref{bignessthrm} is indirectly related to the theory of cohomological decompositions of $\Delta$; we discuss this in Section \ref{blochconjsec}. 

\begin{exmple}
Let $X$ denote the blow-up of $\mathbb{P}^{3}$ along a planar elliptic curve which does not admit complex multiplication.  In Example \ref{blowupelliptic} we verify that $\Delta_{X}$ is big even though $h^{2,1}(X) \neq 0$.  Thus the vanishing results for Hodge groups as in Theorem \ref{bignessthrm} are optimal for threefolds.
\end{exmple}

We emphasize that even amongst varieties satisfying the hypotheses of Theorem \ref{bignessthrm} there are very few with big diagonal.  We will prove several additional strong constraints on the geometry of a variety with big diagonal. For example, such varieties can not admit a morphism to variety with smaller dimension.  Nevertheless, the complete classification of varieties with big diagonal seems subtle (see Section \ref{questions}).

\subsection{Big and nef diagonal}
A cycle class is said to be {\em nef} if it has non-negative intersection against every subvariety of the complementary dimension.  Diagonals which are both big and nef are positive in the strongest possible sense, and we classify such varieties in low dimensions.

\begin{thrm}
Let $X$ be a smooth projective variety.
\begin{itemize}
\item If $\dim X=2$ and $\Delta$ is nef and big then $X$ has the same rational cohomology as $\mathbb{P}^{2}$: it is either $\mathbb{P}^{2}$ or a fake projective plane.
\item If $\dim X=3$ and $\Delta$ is nef and homologically big then $X$ has the same rational cohomology as $\mathbb{P}^{3}$: it is either $\mathbb{P}^{3}$, a quadric, a del Pezzo quintic threefold $V_{5}$, or the Fano threefold $V_{22}$.
\end{itemize}
\end{thrm}

It is interesting to compare this result to Mori's theorem that the only smooth variety with ample tangent bundle is $\mathbb{P}^{n}$.  By switching to the perspective of numerical positivity of $\Delta_{X}$, we also include varieties with the same cohomological properties as projective space.

In higher dimensions we make partial progress toward a classification.  In particular, we show that
$N_{k}(X) \cong \mathbb{R}$ for every $0 \leq k \leq \dim X$, provided that the diagonal is big and {\em universally pseudoeffective} (this is a stronger condition than nefness, in the sense that $\pi^{*}\Delta_{X}$ is required to be pseudoeffective for {\em every} morphism $\pi: Y \to X\times X$).

\subsection*{Dual positivity}

We also study nefness or universal pseudoeffectiveness in the absence of bigness.

\begin{thrm} \label{dualposdiag}
Let $X$ be a smooth projective variety.  If $\Delta_{X}$ is nef (resp.~universally pseudoeffective) then every pseudoeffective class on $X$ is nef (resp.~universally pseudoeffective).
\end{thrm}

For example, a surface with nef diagonal must be a minimal surface.

\begin{exmple}
If $X$ has a nef tangent bundle, then $\Delta_{X}$ is nef.  Campana and Peternell predict that any Fano manifold with nef tangent bundle is in fact rational homogeneous.  Note that Theorem \ref{dualposdiag} is compatible with this conjecture: on a homogeneous variety every pseudoeffective class must be nef.

While varieties with nef tangent bundle will have nef diagonal, the converse is not true; for example, fake projective planes have anti-ample tangent bundle but their diagonals are universally pseudoeffective.\end{exmple}

It is interesting to look for other sources of feedback between nefness of the diagonal and nefness of the tangent bundle.  For example:

\begin{thrm} \label{nefkoddim}
Let $S$ be a smooth surface of Kodaira dimension $\leq 1$ whose diagonal is nef.  Then $T_{S}$ is a nef vector bundle, except possibly when $S$ is a minimal properly elliptic surface with no section.
\end{thrm}

The exception is necessary: Example \ref{hyperellipticnefexample} constructs a hyperelliptic surface with no section which has nef diagonal.  Also, the natural extension to general type surfaces is false: a fake projective plane has nef diagonal.

\subsection*{Examples}

\begin{exmple}[Toric varieties]
Let $X$ be a smooth toric variety.  Theorem \ref{goodproducteff} shows that $\Delta_{X}$ is big if and only if every nef cycle on $X$ is big.  One might expect that the only toric varieties with big diagonal are the projective spaces, but this turns out not to be the case.  For example, \cite{fs09} gives an example of a toric threefold of Picard rank $5$ with big diagonal.

By combining our work with results of \cite{fs09} we can classify toric varieties with nef diagonal:

\begin{prop}
Let $X$ be a smooth projective toric variety.  Then $\Delta_{X}$ is nef if and only if $X$ is a product of projective spaces.
\end{prop}
\end{exmple}

\begin{exmple}[Hypersurfaces]
Let $X$ be a smooth hypersurface of degree $\ge 3$ and dimension $\geq 2$.  It is easy to see that the diagonal of $X$ is not nef.  For bigness, we show:

\begin{thrm}
For a smooth Fano hypersurface of degree $\ge 3$ and dimension $\le 5$, the diagonal is not big.
\end{thrm}
For a quadric hypersurface, $\Delta_X$ is big if and only if the dimension is odd, in which case it is a fake projective space (see Section \ref{quadricsection}).

\end{exmple}

\begin{exmple}[K3 surfaces]
By Theorem \ref{bignessthrm} the diagonal of a K3 surface is not big.  We prove the diagonal of a K3 surface is never nef by using the birational geometry of $\Hilb^{2}(X)$ as described by \cite{bm14}.  For general K3 surfaces we can say more: using a deformation argument we show

\begin{thrm}
For a very general K3 surface, the diagonal is the unique effective $\mathbb{R}$-cycle in its numerical class and it lies on an extremal ray of the pseudoeffective cone.
\end{thrm}

We expect the statement holds for every K3 surface, and we prove it for some specific classes (for example, for K3 surfaces of degree divisible by $4$ with Picard rank $1$).
\end{exmple}

\subsection{Acknowledgements}
We want to thank M.~Fulger for his input and for numerous corrections and improvements. We thank F.~Catanese for a discussion about fake quadrics, E.~Macr{\`{\i}} for a discussion about $\Hilb^{2}$ of K3 surfaces, and X.~Zhao for alerting us to the work of Lie Fu \cite{LieFu}.  BL was supported by an NSA Young Investigator Grant and by NSF grant 1600875, and JCO was supported by RCN grant 250104.

\section{Background} \label{backgroundsec}

Throughout we work over $\mathbb{C}$.  For a projective variety $X$, we will let $\Delta_X$ denote the diagonal in the self-product $X\times X$. The two projections of $X\times X$ will be denoted by $\pi_1$ and $\pi_2$ respectively.

\subsection{Cones of positive cycles}

Let $X$ be a projective variety.  We let $N_{k}(X)_{\mathbb{Z}}$ denote the group of $k$-cycles modulo numerical equivalence.  The numerical class of a cycle $Z$ is written $[Z]$ and we use $\equiv$ to denote numerical equivalence.  The abelian group $N_{k}(X)_{\mathbb{Z}}$ forms a lattice inside the numerical group $
N_{k}(X) := N_{k}(X)_{\mathbb{Z}} \otimes_{\mathbb{Z}} \mathbb{R}$, which is a finite dimensional real vector space.
We define $N^{k}(X)$ to be the vector space dual to $N_{k}(X)$.  When $X$ is smooth of dimension $n$, capping against $[X]$ defines an isomorphism $N^{n-k}(X) \to N_{k}(X)$, and we will switch between subscripts and superscripts (of complementary dimension) freely.

\begin{defn}
We say that a numerical class is {\em effective} if it is the class of an effective $\mathbb{R}$-cycle.  The {\em pseudoeffective cone} $\Eff_{k}(X)$ in $N_{k}(X)$  is the closure of the cone generated by effective classes.  A class is {\em big} when it lies in the interior of $\Eff_{k}(X)$.

The {\em nef cone} $\Nef^{k}(X)$ in $N^{k}(X)$ is the dual of the pseudoeffective cone, and a cycle is called nef if its class belongs to this cone. That is, a cycle is nef if it has non-negative intersection numbers with all $k$-dimensional subvarieties.
\end{defn}

The basic properties of these cones are verified in \cite{fl14}: they are full-dimensional, convex, and contain no lines.  Pseudo-effectiveness is preserved by pushforward, and nefness is preserved by pullback.  It is useful to have more a restrictive form of dual positivity:

\begin{defn}[\cite{fl14}]
Let $X$ be a projective variety.  A cycle class $\alpha \in N^{k}(X)$ is said to be {\em universally pseudoeffective} if $\pi^{*}\alpha $ is pseudoeffective for every morphism $\pi: Y \to X$. %The subcone of $N^{k}(X)$ consisting of universally pseudoeffective classes is denoted $\Upsef^{k}(X)$.
\end{defn}The primary examples of such cycles are complete intersections of ample divisors, or more generally, Chern classes of globally generated vector bundles.
As suggested by the superscript demarcation, the universally pseudoeffective cone is naturally contravariant for morphisms and should be thought of as a ``dual'' positive cone by analogy with the nef cone.

\subsection{Positive homology classes}
Let $H_{2k}(X)_{alg}\subseteq H_{2k}(X)$ denote the subspace of algebraic homology classes, i.e., the image of the cycle class map $cl:CH_k(X)\otimes {\mathbb R}\to H_{2k}(X)$. Let $E_{2k}(X)\subset H_{2k}(X)_{alg}$ denote the {\em cohomological effective cone}.
\begin{defn}
We say a $k$-cycle $\Gamma$ is {\em homologically big} if its cohomology class $[\Gamma]$ lies in the interior of $E_{2k}(X)$.
\end{defn}

In general, for smooth complex projective varieties, cohomological implies numerical equivalence, so any homologically big cycle is big in the usual sense. If Grothendieck's standard conjecture $D$ holds on $X$, namely that numerical and cohomological equivalence coincide, then $N_k(X)=H_{2k}(X)_{alg}$ and the two notions of `big' coincide. In the special case of a self-product, it is known that $D$ holds on $X\times X$ if and only if the Lefschetz standard conjecture holds on $X$ (i.e.~the inverse of the hard Lefschetz isomorphism is induced by a correspondence). This is known to hold for surfaces \cite{lieberman68}, and for threefolds not of general type by results of Tankeev \cite{tankeev}. We will in this paper be mostly interested in surfaces, and use the fact that the two notions coincide in this case without further mention.

%\subsection{Positivity and deformations}\label{B}

We will also require the following result of \cite{Ott15} which follows from the theory of relative Hilbert schemes:

\begin{prop}\label{deformation}
Let $f: \mathcal{X} \to T$ be a smooth family of projective varieties over a smooth variety $T$ and suppose that $\alpha \in H^{k,k}(\mathcal{X},\mathbb{Z})$ has that the restriction to a very general fiber is represented by an effective cycle.  Then $\alpha|_{\mathcal{X}_{t}}$ is an effective class for any fiber $\mathcal{X}_{t}$.
\end{prop}
We can use this result when $\mathcal{X}$ is a family of varieties for which homological and numerical equivalence coicide (e.g., fourfolds). In this case, the theorem also implies that a class which restricts to be big on a very general fiber has big restriction on every fiber.

\section{Varieties with big diagonal}
In this section we consider the geometric implications of big diagonals.  %$\Delta_{X}\subset X\times X$ being big, so that $[\Delta_{X}]$ lies in the interior of the effective cone ${\operatorname{Eff}}_n(X\times X)$. 

\begin{lem} \label{bigimpliesbigupsef}
Let $X$ be a projective variety.  If $X$ carries a universally pseudoeffective class $\alpha \in N^{k}(X)$ that is not big, then $\Delta_{X}$ is not big.
\end{lem}

In particular, if $X$ carries a nef divisor that is not big, then $\Delta_{X}$ is not big.

\begin{proof}
Let $n$ denote the dimension of $X$.  Since $\alpha$ is not big, there is some non-zero nef class $\beta \in N^{n-k}(X)$ that has vanishing intersection with $\alpha$.  Then consider $\gamma := \pi_{1}^{*}\alpha \cdot \pi_{2}^{*}\beta$ on $X \times X$.  Clearly $\gamma$ is a nef class: if $E$ is an effective cycle of dimension $n$, then $\pi_{1}^{*}\alpha \cdot E$ is still pseudoeffective, so that it has non-negative intersection against the nef class $\pi_{2}^{*}\beta$.  Since $\gamma \cdot \Delta_{X} = 0$, we see that $\Delta_{X}$ can not be big.
\end{proof}

\begin{cor} \label{bigmorcor}
Let $X$ be a projective variety of dimension $n$.  If $X$ admits a surjective morphism $f: X \to Y$ to a variety of dimension $< n$, then $\Delta_{X}$ is not big.
\end{cor}

\def\K{\mathcal K}

It is sometimes helpful to consider non-algebraic classes as well. In this setting, we recall that a $(1,1)$-cohomology class $\alpha$ is defined to be nef if it is the limit of K\"ahler classes.

\begin{thrm}\label{analytictheorem}
Let $X$ be an $n$-dimensional smooth projective variety admitting a non-zero nef cohomology class $\alpha \in H^{1,1}(X,\mathbb{R})$ such that $\alpha^n=0$. Then $\Delta_{X}$ is not homologically big.
\end{thrm}
\begin{proof}Let $\omega$ be a K\"ahler form on $X\times X$. Let $\alpha$ be a nef $(1,1)$-form on $X$ and let $0< k< n$ be an integer so that $\alpha^k\neq 0$, but $\alpha^{k+1}= 0$. The two pullbacks $\pi_1^*\alpha^k$ and $\pi_2^*\alpha \cup \omega^{n-k-1}$ are weakly positive forms on $X\times X$, and hence their product $$\beta=\pi_1^*\alpha^k\cup \pi_2^*\alpha\cup \omega^{n-k-1}$$ is a weakly positive $(n,n)$-class on $X\times X$ \cite[Ch. III]{dembook}. Now the main point is that $\beta$ is nef, in the sense that $\int_Z \beta\ge 0$ for all subvarieties $Z\subset X \times X$. This is because $\beta$ restricts to a non-negative multiple of the volume form on $Z$ for every smooth point on it (cf.  \cite[Ch. III (1.6)]{dembook}). Note however that it is not in general the case that the product of two nef classes remains nef,  as shown in \cite{delv11}. 

If $\Delta_{X}$ is homologically big, then we can write $[\Delta_{X}] = \epsilon h^n+Z$ where $\epsilon>0$, $h$ is an ample line bundle and $Z$ is an effective cycle. Moreover, since $h$ is ample and $\alpha$ is nef, the following two inequalities hold:
\begin{equation}\label{cupprod}\int_Z \beta\ge 0 \,\,\mbox { and }\int_X h^n \cup \beta>0
\end{equation}
However, these contradict $\beta\cdot \Delta_{X}=0$, which holds by our assumptions on $\alpha$ and $k$. \end{proof}

Bigness of the diagonal is compatible with pushforward:

\begin{lem} \label{pushforwardpreservesbigdiag}
Let $f: X \to Y$ be a surjective morphism of projective varieties.  If $\Delta_{X}$ is big, then so is $\Delta_{Y}$.
\end{lem}

Note that by Lemma \ref{bigimpliesbigupsef} the hypothesis is never satisfied if $0<\dim Y < \dim X$, so the main interest is in the generically finite case.

\begin{proof}
Let $n$ denote the dimension of $X$ and $d$ denote the dimension of $Y$.  Fix an ample divisor $H$ on $X$.  Then $\Delta_{X} \cdot H^{2n-2d}$ is a big class on $X$.

Consider the induced map $f \times f: X \times X \to Y \times Y$.  The set-theoretic image of $\Delta_{X}$ is $\Delta_{Y}$; in particular,  $(f \times f)_{*}: [\Delta_{X}] \cdot H^{2n-2d}$ is proportional to $\Delta_{Y}$.  Since the pushforward of a big class under a surjective map is still big, we see that $\Delta_{Y}$ is also big.
\end{proof}

%We can also rule out the bigness of $\Delta_{X}$ in the presence of many rigid cycles.  We say that a cycle is numerically rigid if it is the unique effective cycle in its numerical class.

%\begin{prop}
%Let $X$ be a smooth surface containing infinitely many numerically rigid irreducible curves.  Then $\Delta_{X}$ is not big.  
%\end{prop}

%\begin{proof}
%Suppose $\Delta_{X}$ is big.  We can write $\Delta_{X} \equiv A+E$ where $A$ is proportional to a complete intersection of very ample Cartier divisors and $E$ is an effective $\mathbb{R}$-cycle.  Consider intersecting $A+E$ against $\pi_{1}^{*}D$ for a numerically rigid divisor $D$.  Since there are infinitely many such $D$ by hypothesis, by choosing a sufficiently general $D$ we may suppose that the only components of $\Supp(E)$ contained in $\pi_{1}^{-1}(D)$ are actually fibers of $\pi_{1}$ (which a priori can lie simultaneously in the intersection of all the $\pi_{1}^{-1}(D)$).  Note that such components of $E$ can be moved off of $\pi_{1}^{-1}(D)$.  After this change both intersections $A \cdot \pi_{1}^{*}D$ and $E \cdot \pi_{1}^{*}D$ are represented by effective cycles.  Furthermore,  by choosing $A$ sufficiently general we may ensure that $\pi_{2*}(A \cdot \pi_{1}^{*}D)$ pushes forward under $\pi_{2}$ to a divisor $D'$ whose support is different from $D$.  But this contradicts the numerical rigidity of $D$.
%\end{proof}

\subsection{Cohomological criteria}
The main result of this section is the following theorem.

\begin{thrm}\label{qpg}
Let $X$ be a smooth projective variety with homologically big diagonal. Then $H^{k,0}(X)=0$ for all $k>0$.
\end{thrm}

In particular, no varieties with trivial canonical bundle can have homologically big diagonal.

\begin{proof}
Following \cite{coniveau} and \cite{LieFu}, we will utilize the Hodge--Riemann relations to find faces of the effective cones of cycles.
To set this up, let $\omega$ be a K\"ahler form on a smooth projective variety $W$.  Note that a cohomology class in $H^{k,0}(W)$ is automatically primitive.  Thus by the Hodge--Riemann bilinear relations, the bilinear form on $H^{k,0}(W)$ given by $$q(a,b)=\varepsilon \int_W a \cup \bar{b}\cup \omega^{n-k}$$ is positive definite. Here $\varepsilon=1$ if $k$ is even, and $\varepsilon=\sqrt{-1}$ if $k$ is odd. 

Now fix a K\"ahler form $\omega$ on $X \times X$ and let $\sigma$ be a non-zero closed $(k,0)$-form on $X$. Consider the product 
$$
\beta=\varepsilon\left(\pi_1^*\sigma -\pi_2^*\sigma\right)\cup \left(\pi_1^*\bar \sigma-\pi_2^*\bar\sigma\right)\cup \omega^{n-k}.
$$This is a non-zero $(n,n)$-form on $X\times X$, which by construction vanishes on the diagonal. 

Now, if $Z\subset X\times X$ is an $n$-dimensional subvariety, the Hodge--Riemann relations (applied on a resolution of $Z$) imply that $\beta \cdot Z\ge 0$. Similarly, $\beta \cdot h^{n}>0$ for an ample divisor $h$ on $X\times X$. Finally, since $\Delta_X\cdot \beta=0$, it follows that $\Delta$ cannot be homologically big.
\end{proof}

\begin{rmk}
The above theorem can also be deduced from \cite[Lemma 3.3]{LieFu}, which is proved using a similar argument. 
\end{rmk}

%Now Theorem \ref{qpg} will be a consequence of the following lemma.

%\begin{lem}
%The cohomology class $\beta$ is nef, i.e., $\int_Z\beta\ge 0$ for every subvariety $Z\subset X\times X$.
%\end{lem}
%\begin{proof}
%Let $\phi:W\to Z$ be a desingularization of $Z$, then $\phi^*(\beta)$ is of the form $\epsilon \cdot w\cup \bar w \cup (\pi^*\omega)^{n-k}$, where $w=\pi^*\alpha$ is a $(k,0)$-class, hence primitive. Since $\pi^*\omega$ is a limit of K\"ahler classes, integrating the above form over $W$ yields a non-negative number by the Hodge bilinear relations, and so $\beta$ is nef.
%\end{proof}

%The above proof, using the Hodge--Riemann relations, was inspired by Voisin's paper \cite{coniveau}. See also Lemma 3.3 in Fu's paper \cite{LieFu} for a similar result.

\begin{exmple}
Even when the diagonal is only (numerically) big, we can still show that $H^{1,0}(X)$ vanishes.  First suppose that $A$ is an abelian variety of dimension $n$.  Then the diagonal is the fiber over 0 of the subtraction map $f: A \times A \to A$.  In particular, $\Delta_{A}$ has vanishing intersection against the nef class $f^{*}L \cdot H^{n-1}$ where $L$ is an ample divisor on $A$ and $H$ is an ample class on $A \times A$.  Under suitable choices, the class $\beta$ in the proof of Theorem \ref{qpg} constructed from $H^{1,0}(A)$ will be exactly this $n$-cycle.

More generally, when $X$ is a smooth projective variety with non-trivial Albanese, we have a subtraction map $X \times X \to A$.  The diagonal will have vanishing intersection against the pullback of an ample divisor from $A$ under the subtraction map intersected with an appropriate power of an ample divisor on $X \times X$.  Again, this is essentially the same as the class $\beta$ constructed in the proof above.
\end{exmple}

\begin{exmple} \label{blowupelliptic}
We give an example of a smooth Fano threefold $X$ with homologically big diagonal which satisfies $h^{2,1}(X) \neq 0$.  Thus Theorem \ref{qpg} is optimal in the sense that the other Hodge groups need not vanish.

Let $X$ be the blow-up of $\mathbb{P}^{3}$ along a planar elliptic curve $C$ which does not have complex multiplication.  Let $H$ denote the pullback of the hyperplane class to $X$ and $E$ denote the exceptional divisor.  It is easy to verify that:
\begin{align*}
\Eff_{2}(X) = \langle H-E, E \rangle \qquad \qquad & \Nef_{2}(X) = \langle 3H-E,H \rangle  \\
\Eff_{1}(X) = \langle HE, H^{2}-HE \rangle \qquad \qquad & \Nef_{1}(X) = \langle H^{2}, 3H^{2} - HE \rangle
\end{align*}

On $X \times X$ let $H_{i},E_{i}$ denote the pullbacks of $H$ and $E$ under the $i$th projection.  Since $C$ does not have complex multiplication, $N_{3}(X \times X)$ has dimension $11$: it is spanned by $\Delta_{X}$ and the non-zero products of $H_{1},E_{1},H_{2},E_{2}$.

Recall that $C \times C$ has three-dimensional Neron-Severi space spanned by the fibers $F_{1},F_{2}$ of the projections and the diagonal $\Delta_{C}$.  Let $Z_{a,b,c}$ denote the class in $N_{3}(X \times X)$ obtained by pulling the divisor $aF_{1} + bF_{2} + c\Delta_{C}$ back from $C \times C$ to $E \times E$ and then pushing forward to $X \times X$.  An intersection calculation shows that
\begin{equation*}
Z_{a,b,c} = \frac{a}{3}H_{1}E_{1}E_{2} + \frac{b}{3}H_{2}E_{1}E_{2} + c ( H_{1}^{3} + H_{1}^{2}H_{2} + H_{1}H_{2}^{2} + H_{2}^{3} - \Delta_{X}).
\end{equation*}
Applying this to the effective divisor $2F_1+2F_2-\Delta$, we obtain
\begin{align*}
\Delta_{X} = & Z_{2,2,-1} + \frac{1}{6}H_{1}E_{1}E_{2} + \frac{1}{6}H_{2}E_{1}E_{2} +  \frac{5}{6} H_{1}E_{1}(H_{2} - E_{2}) + \frac{5}{6}H_{2}E_{2}(H_{1} - E_{1}) \\
& + \frac{5}{6}H_{1}H_{2}(H_{1} - E_{1}) + \frac{5}{6}H_{1}H_{2} (H_{2} - E_{2}) +  \frac{1}{6}H_{1}^{2}(H_{2} - E_{2}) + \frac{1}{6}H_{2}^{2}(H_{1} - E_{1}) \\
& + \frac{1}{6}H_{1}^{2}E_{2} + \frac{1}{6}H_{2}^{2}E_{1} +  H_{1}^{3} + H_{2}^{3}
\end{align*}
and since the terms are all effective and together span $N_{3}(X \times X)$ we see $\Delta_{X}$ is big (and hence homologically big, since $X$ is a rational threefold).
We also note in passing that $\Delta_{X}$ is not nef, since it has negative intersection against the effective cycle $H_{1}E_{1}E_{2}$.
\end{exmple}

\subsection{Criteria for bigness}

There is one situation where it is easy to test for bigness of the diagonal, namely when the effective cones of $X\times X$ are as simple as possible.

\begin{thrm} \label{goodproducteff}
Let $X$ be a smooth projective variety of dimension $n$.  Suppose that for every $k$
\begin{equation*}
\Eff_{k}(X \times X) = \sum_{i+j=k} \pi_{1}^{*}\Eff_{i}(X) \cdot \pi_{2}^{*} \Eff_{j}(X).
\end{equation*}
Then $\Delta_{X}$ is big if and only if every nef class on $X$ is big.  
\end{thrm}

\begin{proof}
We first claim that the nef cone has the expression
\begin{equation*}
\Nef^{k}(X \times X) = \sum_{i+j=k} \pi_{1}^{*}\Nef^{i}(X) \cdot \pi_{2}^{*} \Nef^{j}(X).
\end{equation*}
The containment $\supseteq$ is clear from the description of the pseudoeffective cone.  Conversely, it suffices to show that every class generating an extremal ray of $\Eff_{k}(X \times X)$ has vanishing intersection against some element of the right hand side.  By hypothesis such classes have the form $\pi_{1}^{*}\alpha_{i} \cdot \pi_{2}^{*}\alpha_{k-i}$ where $\alpha \in \Eff_{i}(X)$ and $\alpha_{k-i} \in \Eff_{k-i}(X)$ both lie on extremal rays.  Choose nef classes $\beta^{i} \in \Nef^{i}(X)$ and $\beta^{k-i} \in \Nef^{k-i}(X)$ satisfying $\alpha_{i} \cdot \beta^{i} = 0$ and $\alpha_{k-i} \cdot \beta^{k-i} = 0$.  Then
\begin{equation*}
(\pi_{1}^{*}\alpha_{i} \cdot \pi_{2}^{*}\alpha_{k-i}) \cdot (\pi_{1}^{*}\beta^{i} \cdot \pi_{2}^{*}\beta^{n-i}) = 0
\end{equation*}

Now suppose that $\Delta_{X}$ is not big.  Then it must have vanishing intersection against some $\alpha \in \Nef^{n}(X \times X)$ which lies on an extremal ray.  By the expression above, such a class has the form\begin{equation*}
\alpha = \pi_{1}^{*}\beta_{j} \cdot \pi_{2}^{*}\beta_{n-j}
\end{equation*}
where for some constant $j$ we have $\beta_{j} \in \Nef^{j}(X)$ and $\beta_{n-j} \in \Nef^{n-j}(X)$.  But then $\beta_{j} \cdot \beta_{n-j} = 0$ as classes on $X$.  Since $\beta_{j}$ has vanishing intersection against a nef class, it can not be big.

Conversely, suppose that there is a nef class in $N_{k}(X)$ which is not big.  Since there are also big nef classes in $N_{k}(X)$, by convexity of the nef cone we can find a nef class $\beta \in N_{k}(X)$ on the boundary of the pseudoeffective cone.  Thus there is another nef class $\beta'$ such that $\beta \cdot \beta' = 0$.  Arguing as above, we see that $\pi_{1}^{*}\beta \cdot \pi_{2}^{*}\beta'$ is a nef class with vanishing intersection against $\Delta_{X}$.
\end{proof}

Two typical situations where one can apply Theorem \ref{goodproducteff} are when:

\begin{itemize}
\item $X$ is a toric variety.
\item $N_{k}(X \times X) = \oplus_{i+j=k}\pi_{1}^{*}N_{i}(X) \cdot \pi_{2}^{*} N_{j}(X)$, every pseudoeffective cone on $X$ is simplicial, and every nef class on $X$ is universally pseudoeffective.
\end{itemize}

The first fact is well-known.  To see the second, note that the hypothesis on universal pseudo-effectivity shows that any external product of nef cycles is nef.  The simplicial hypothesis then implies that the external product of the pseudoeffective cones is dual to the external product of the nef cones.  Thus the external product of the pseudoeffective cones is in fact the entire pseudoeffective cone of $X \times X$.

We will apply Theorem \ref{goodproducteff} to examples where one can prove directly that all nef classes are universally pseudoeffective (e.g., fake projective spaces, Grassmannians,\ldots). However, it seems relatively rare in general for the condition on pseudoeffective cones in Theorem \ref{goodproducteff} to hold. Here is a basic example:

\begin{exmple} \label{effdecomofdiag}
Let $S$ be the blow-up of $\mathbb{P}^{2}$ in $r$ general points for some $r \geq 5$. There is a strict containment
\begin{equation*}
\mathbb{R}_{\geq 0}[F_{1}] \oplus \pi_{1}^{*}\Eff_{1}(S) \cdot \pi_{2}^{*}\Eff_{1}(S) \oplus \mathbb{R}_{\geq 0}[F_{2}] \subsetneq \Eff_{2}(S \times S).
\end{equation*}
In fact, a lengthy but straightforward computation shows that the diagonal does not lie in the cone on the left. 
\end{exmple}

\section{Dual positivity}

We next turn to the ``dual'' forms of positivity: nefness and universal pseudoeffectiveness.  The main examples are varieties with nef tangent bundle.  For such varieties the class of $\Delta_{X}$ is nef, but not all varieties with nef diagonal have nef tangent bundle; for example, a fake projective plane has nef diagonal even though the tangent bundle is antiample.

We emphasize that only ``dual-positivity'' of the tangent bundle should be inherited by the diagonal.  The bigness of the tangent bundle $T_{X}$ is quite different from the bigness of the class $\Delta_{X}$.  For example, a product of at least two projective spaces has big and nef tangent bundle, but by Lemma \ref{bigimpliesbigupsef} the diagonal class is not big.  More generally, a smooth toric variety has big tangent bundle by \cite{hsiao15}, but it is rare for a toric variety to have big diagonal.

\begin{prop} \label{nefimpliesnefcycles}
Let $X$ be a smooth variety.  If $\Delta_{X}$ is nef (resp.~universally pseudoeffective) then every pseudoeffective class on $X$ is nef  (resp.~universally pseudoeffective).
\end{prop}

In fact, the proposition is true for any property preserved by pullback and flat pushforward.  This proposition strengthens \cite[Proposition 2.12]{cp91}, which shows the analogous statement for divisors on a variety with $T_{X}$ nef.

\begin{proof}
We focus on nefness; the proof for universal pseudoeffectiveness is identical, using the properties of positive dual classes proved in \cite{fl14}.

It suffices to show nefness for the class of an irreducible cycle $Z$ on $X$.  Since $\pi_1$ is flat, $Z' = \pi_{1}^{-1}(Z)$ represents $\pi_{1}^{*}[Z]$.  The restriction of $\Delta_{X}$ to $Z'$ is nef; since nefness is preserved by flat pushforward onto a smooth base, $(\pi_{2}|_{Z'})_{*} [\Delta_{X}]|_{Z'} = [Z]$ is also nef on $X$.
\end{proof}

\begin{cor} \label{coro:upsefrank} Let $X$ be a smooth projective variety.
\begin{enumerate}
\item If $\Delta_{X}$ is big and nef, then $N^{1}(X) \cong \mathbb{R}$.
\item If $\Delta_{X}$ is big and universally pseudoeffective, then $N_{k}(X) \cong \mathbb{R}$ for every $k$.
\end{enumerate}
\end{cor}

\begin{proof}
Combine Proposition \ref{nefimpliesnefcycles}, Lemma \ref{bigimpliesbigupsef}, and the fact that nef divisors are universally pseudoeffective.
\end{proof}

\begin{lem} \label{neflowgenus}
If a smooth variety $X$ admits a surjective map to a curve $C$ of genus $\geq 2$, then $\Delta_{X}$ is not nef.
\end{lem}

\begin{proof}
Denote the morphism by $\pi: X \to C$.  Let $H$ be an ample divisor on $X \times X$.  Letting $n$ denote the dimension of $X$, we have $\Delta_{X} \cdot H^{n-1} \cdot (\pi \times \pi)^{*}\Delta_{C} < 0$ by the projection formula and the fact that $\Delta_{C}$ has negative self-intersection.
\end{proof}

We can also give a necessary condition for nefness based on the gonality of $X$.

\begin{prop}
Let $X$ be a smooth projective variety of dimension $n$ admitting a surjective generically finite map $f: X \to Y$ of degree $d$ to a smooth projective variety $Y$. Suppose that $c_{n}(X) > dc_{n}(Y)$.  Then $\Delta_{X}$ is not nef.
\end{prop}

\begin{proof}
If $f$ contracts a curve, then $X$ carries a curve that is not nef, and hence $\Delta_{X}$ is not nef.  Thus it suffices to consider the case when $f$ is finite.

Consider the map $F=(f\times f):X\times X\to Y \times Y$. This is finite surjective, hence flat. Note that $F|_{\Delta_{X}}=f$, so $F_*\Delta_X=d \Delta_{Y}$. Moreover, by flatness, $F^*\Delta_{Y}$ is an effective cycle containing $\Delta_{X}$ in its support. The intersection of $F^*\Delta_{Y}-\Delta_X$ with $\Delta_X$ is $$
F_*\Delta_X\cdot \Delta_{Y}-\Delta_X^2=dc_{n}(Y)-c_n(X)
$$which is negative by assumption, so that $\Delta_{X}$ is not nef.
\end{proof}

\begin{cor} \label{gonalitynefness}
Let $X$ be a smooth projective variety of dimension $n$ admitting a surjective generically finite map $f: X \to \mathbb{P}^{n}$ of degree $d$. Suppose that $c_{n}(X) > (n+1)d$.  Then $\Delta_{X}$ is not nef.
\end{cor}

\section{Rigidity} \label{rigiditysection}
The results of the previous sections indicate that it is quite rare for a variety to have big diagonal. In this section we will study varieties where $\Delta_X$ is as far away from big as possible, and in particular, when $[\Delta_X]$ spans an extremal ray in the pseudoeffective cone.

%The following definition describes the effective cycles which are as far away from big as possible.

\begin{defn} \label{rigiddef}
Let $Z$ be an effective $\mathbb{R}$-cycle on a projective variety $X$ of dimension $k$.  We say that $Z$ is:
\begin{enumerate}
\item {\em strongly numerically rigid}, if $Z$ is irreducible and for every infinite sequence of effective $\mathbb{R}$-cycles $Z_{i}$ such that $\lim_{i \to \infty} [Z_{i}] = [Z]$, the coefficient $a_{i}$ of $Z$ in $Z_{i}$ limits to $1$.
\item {\em exceptional for a morphism $\pi:X \to Y$}, if $\reldim(\pi|_{Z}) > \reldim(\pi)$.
\end{enumerate}
\end{defn}

Exceptional classes are studied in \cite{fl15b} and are closely related to the notion of an exceptional divisor.  Among other nice properties, an exceptional numerical class can not be represented by a cycle whose deformations cover $X$.  

If $Z$ is strongly numerically rigid then it spans an extremal ray of the pseudoeffective cone and is the unique effective cycle in its numerical class.  A typical example of a strongly numerically rigid class is an irreducible divisor of numerical dimension $0$.  A related concept is discussed briefly in \cite[Page 93 Remark]{nakayama04}.

\subsection{Blowing up}

\begin{lem} \label{strongrigiditylem}
Let $X$ be a smooth projective variety and let $Z$ be an $k$-dimensional subvariety.  Suppose that there is an open neighborhood $U \subset N_{k}(X)$ of $[Z]$ such that $Z$ appears with positive coefficient in any effective $\mathbb{R}$-cycle with class in $U$.  Then $Z$ is strongly numerically rigid.
\end{lem}

The point is that there is no assumed lower bound for the coefficient with which $Z$ appears in the cycles.

\begin{proof}
We first show that, perhaps after shrinking $U$, there is a constant $\epsilon > 0$ such that $\epsilon Z \leq T$ for any effective $\mathbb{R}$-cycle $T$ with numerical class in $U$.  Suppose otherwise for a contradiction.  Choose $\beta$ in the interior of the movable cone (that is, the closure of the cone of classes subvarieties which deform to cover $X$).  For some sufficiently small $\tau$ we have that $\beta + \tau [Z]$ is still in the interior of the movable cone.  Thus, if $\alpha \in U$ has an effective representative where $Z$ appears with coefficient $c$, the class $\alpha + \frac{c}{\tau} \beta$ is represented by an effective $\mathbb{R}$-cycle in which $Z$ has coefficient $0$.  If there is an open neighborhood $U'$ of $[Z]$ with $\overline{U'} \subset U$ and admitting representatives with arbitrarily small coefficients of $Z$, we obtain a contradiction.

We can now argue as in \cite[Page 93 Remark]{nakayama04}: we define a function $\sigma_{Z}: \Eff^{\circ}_{k}(X) \to \mathbb{R}$ that records the infimum of the coefficients of $Z$ appearing in any effective $\mathbb{R}$-cycle of class $\alpha$.  This function is continuous on the big cone; by taking limits we extend it to a lower semicontinuous function on the entire pseudoeffective cone.  Furthermore, for any $\alpha \in \Eff_{k}(X)$ and $\beta$ in the interior of the movable cone, the restriction of $\sigma_{Z}$ to the ray $\alpha + t\beta$ is strictly decreasing in $t$.  We deduce that $\sigma_{Z}([Z]) > 0$.  An easy rescaling argument shows that $\sigma_{Z}([Z]) = 1$, and we conclude the strong numerical rigidity of $[Z]$ by lower semi-continuity of $\sigma_{Z}$.
\end{proof}

We can then test for the strong numerical rigidity of $\Delta$ by blowing up $\Delta$.

\begin{prop} \label{blowupcomputation}
Let $X$ be a smooth projective variety of dimension $n$.  Let $\phi: W \to X \times X$ denote the blow-up of the diagonal and let $i: E \to W$ denote the inclusion of the exceptional divisor.  Suppose that $\alpha \in N_{n}(X \times X)$ is a non-zero class such that
\begin{equation*}
\phi^{*}\alpha = M + i_{*}N
\end{equation*}
where $M \in N_{n}(W)$ is a nef class and $N \in N_{n}(E)$ is a nef class.
\begin{enumerate}[(1)]
\item If $\alpha \cdot \Delta_{X} = 0$ then $\Delta_{X}$ is not big.
\item If $\alpha \cdot \Delta_{X} < 0$ then $\Delta_{X}$ is strongly numerically rigid.
\end{enumerate}
\end{prop}

\begin{proof}
Let $T$ be an effective $n$-cycle on $W$.  If $T$ is not supported on $E$ then $T \cdot \phi^{*}\alpha$ is non-negative.  Pushing forward, we see that the only effective $n$-cycle on $X$ which can possibly have negative intersection with $\alpha$ is $\Delta_{X}$ itself.  Thus:

(1) Suppose $\alpha \cdot \Delta_{X} = 0$.  Then $\alpha$ is nef and thus $\Delta_{X}$ can not be big.

(2) Suppose $\alpha \cdot \Delta_{X} < 0$.  Then also $\alpha \cdot \beta < 0$ for any effective class $\beta$ sufficiently close to $[\Delta_{X}]$.  This means that any effective representative of such a $\beta$ must contain $\Delta_{X}$ in its support with positive coefficient.  We conclude that $\Delta_{X}$ is strongly numerically rigid by Lemma \ref{strongrigiditylem}.
\end{proof}

For surfaces, we have the following criterion:

\begin{prop} \label{blowuprigidity}
Let $S$ be a smooth surface.  Let $\phi: Y \to S\times S$ denote the blow-up of $\Delta_{S}$.  If $\phi^{*}[\Delta_{S}]$ is not pseudoeffective, then $\Delta_{S}$ is strongly numerically rigid.
\end{prop}

\begin{proof}
We let $E$ denote the exceptional divisor of $\phi$ and let $g: E \to \Delta_{S}$ denote the projective bundle map and $\xi$ the class of the relative $\mathcal{O}(1)$ on $E$.  We denote by $i: E \to Y$ the inclusion.

Suppose that $\phi^{*}[\Delta_{S}]$ is not pseudoeffective, and let $\eta$ be a nef class in $N^2(Y)$ such that $\eta \cdot \phi^{*}[\Delta_{S}] < 0$.  Choose a sufficiently small open subset $U \subset N_2(S\times S)$ of $[\Delta_{S}]$ such that $\eta \cdot \phi^{*}\beta < 0$ for every $\beta \in U$.  Let $Z$ be any effective $\mathbb{R}$-cycle on $S\times S$ such that $[Z] \in U$.  Let $T$ be any effective $\mathbb{R}$-cycle on $Y$ that pushes forward to $Z$; after removing vertical components, we may suppose that $T$ does not have any components contracted by $\phi$.  Let $\alpha$ denote the class of $T$.  We can write $\alpha = \phi^{*}\phi_{*}\alpha + i_{*}g^{*}L$ for some (not necessarily effective) $\mathbb{R}$-divisor class $L$ on $S$.  Then since $\eta\cdot \alpha\ge 0$, and $\eta\cdot \phi^*\phi_*\alpha<0$, we have $\eta \cdot i_{*}g^{*}L> 0$ and consequently
$$g_{*}(\eta|_{E}) \cdot L > 0.$$
Now, since $\eta|_{E}$ is nef and $g$ is flat, $g_{*}(\eta|_{E})$ is the class of a nef curve $\widetilde{\eta}$ on $S$.  Then, if $\pi_1$ denotes the projection to the first factor, we find  
\begin{align*}
E \cdot \phi^{*}\pi_{1}^{*}\widetilde{\eta} \cdot \alpha & =  E \cdot \phi^{*}\pi_{1}^{*}\widetilde{\eta} \cdot i_{*}g^{*}L \\
& = (-\xi) \cdot g^{*}(\widetilde{\eta} \cdot L)  < 0.
\end{align*}
By the nefness of $\widetilde{\eta}$ (and hence $\phi^{*}\pi_{1}^{*}\widetilde{\eta}$), we see that some component of $T$ must be contained in $E$, and furthermore (since we removed all $\pi$-contracted components) this component must dominate $\Delta_{S}$ under $\pi$.
Pushing forward, we see that $\Delta_{S}$ must be contained in $Z$ with positive coefficient.  We conclude by Lemma \ref{strongrigiditylem}.
\end{proof}

\subsection{Rigidity via the Hilbert scheme} \label{hilbschemesection}

Using the rational map $S \times S \dashrightarrow \Hilb^{2}(S)$, one can study the positivity of $\Delta_{S}$ via the geometry of the Hilbert scheme.  This approach is surprisingly successful, allowing us to use results arising from Bridgeland stability.

\begin{thrm} \label{hilb2thrm}
Let $S$ be a surface and let $B'$ denote the divisor on $\Hilb^{2}(S)$ such that $2B'$ parametrizes non-reduced subschemes.  For nef divisors $H$ and $A$ on $X$, consider $D_{1} := H^{[2]} - b_{1}B'$ and $D_{2} := A^{[2]} - b_{2}B'$ on $\Hilb^{2}(S)$.  If $c_2(S) > 0$ and
\begin{itemize}
\item $D_{1}$ and $D_{2}$ are movable and
\begin{equation*}
b_{1}b_{2} > \frac{4A \cdot H}{c_{2}(S)}
\end{equation*}
then $\Delta_{S}$ is not nef.
\item $D_{1}$ is nef, $D_{2}$ is movable, and
\begin{equation*}
b_{1}b_{2} > \frac{4A \cdot H}{c_2(S)}
\end{equation*}
then $\Delta_{S}$ is strongly numerically rigid.
\end{itemize}
\end{thrm}

\begin{proof}
Let $\phi: Y \to S \times S$ be the blow-up along the diagonal.  The exceptional divisor $E$ is isomorphic to $\mathbb{P}(\Omega^{1}_{S})$ with projection $g: E \to S$.  Letting $\xi$ denote the class of the relative $\mathcal{O}(1)$ and $i: E \to Y$ the injection, we have that $\phi^{*}\Delta_{S} = i_{*}(\xi - g^{*}K_{S})$.

Let $\psi: Y \to \Hilb^{2}(S)$ denote the $2:1$-map.  Then we  compute intersections by restricting to $E$:
\begin{align*}
\psi^{*}D_{1} \cdot \psi^{*}D_{2} \cdot \phi^{*}\Delta_{S} & = (2g^{*}H+b_{1}\xi) \cdot (2g^{*}A+b_{2}\xi) \cdot (\xi - g^{*}K_{S}) \\
& = -b_{1}b_{2}c_2(S) + 4 A \cdot H
\end{align*}

First suppose that $D_{1}$ and $D_{2}$ are movable and the inequality holds.  Since $\psi$ is finite, $\psi^{*}D_{1}$ and $\psi^{*}D_{2}$ are also movable, and hence their intersection is pseudoeffective.  The assumed inequality shows that $\psi^{*}D_{1} \cdot \psi^{*}D_{2} \cdot \phi^{*}\Delta_{S} < 0$, so that $\Delta$ is not nef.

Next suppose that $D_{1}$ is nef and $D_{2}$ is movable.  Then $\psi^{*}D_{1} \cdot \psi^{*}D_{2}$ is nef and by the same calculation as before we deduce that $\phi^{*}\Delta_{S}$ is not pseudoeffective.  By Proposition \ref{blowuprigidity}  $\Delta_{S}$ is strongly numerically rigid.
\end{proof}

It would be interesting if Theorem \ref{hilb2thrm} could be improved by a more in-depth study of the geometry of the Hilbert scheme $\Hilb^{2}(S)$.

\subsection{Albanese map} \label{albanesesection}

Let $X$ be a smooth projective variety and let $alb: X \to A$ be the Albanese map (for a chosen basepoint).  By the subtraction map for $X$, we mean the composition of $alb^{\times 2}: X \times X \to A \times A$ with the subtraction map for $A$.  Note that this map does not depend on the choice of basepoint.

\begin{prop}
Let $X$ be a smooth projective variety of dimension $n$.  Suppose that the Albanese map $alb: X \to A$ is generically finite onto its image but is not surjective.  Then $\Delta_{X}$ is exceptional for the subtraction map.
\end{prop}

\begin{proof}
Note that the diagonal is contracted to a point by the subtraction map.  Thus, it suffices to prove that a general fiber of the subtraction map $f: X \times X \to A$ has dimension $< n$.  Let $X'$ denote the image of the albanese map.  Since $alb$ is generically finite onto its image, it suffices to prove that the general fiber of the subtraction map $f: X' \times X' \to A$ has dimension $< n$.

Suppose otherwise for a contradiction.   For every closed point $p \in f(X' \times X')$ the fiber $F_{p}$ denotes pairs of points $(x_1,x_2) \in X' \times X'$ such that $x_1 = p + x_2$.  If this has dimension $n$, then it must dominate $X'$ under both projections.  In other words, $X'$ is taken to itself under translation by every point of $f(X' \times X')$.  Recall that $X'$ contains the identity of $A$, so that in particular $X' \subseteq f(X' \times X')$.  Thus, the subgroup of $A$ fixing $X'$ is all of $A$.  This is a contradiction when $X' \neq A$.
\end{proof}

There are many other results of a similar flavor.  For example, if the diagonal is the only subvariety of dimension $\geq n$ contracted by the Albanese map then $\Delta_{X}$ is strongly numerically rigid using arguments similar to those of \cite[Theorem 4.15]{fl15b}.  This situation holds for every curve of genus $\geq 2$ and seems to hold often in higher dimensions as well.

\section{Surfaces}

We now discuss positivity of the diagonal for smooth surfaces.  First, by combining Theorem \ref{qpg} with Corollary \ref{coro:upsefrank} (and using the equality of homological and numerical equivalence for surface classes) we obtain: 

\begin{thrm}
The only smooth projective surfaces with big and nef diagonal are the projective plane and fake projective planes.
\end{thrm}

In this section we discuss each Kodaira dimension in turn.  We can summarize the discussion as follows:

\begin{itemize}
\item The only possible surfaces with big diagonal are $\mathbb{P}^{2}$ or a surface of general type satisfying $p_{g} = q = 0$.  In the latter case, the only example with big diagonal that we know of is a fake projective plane.
\item If the Kodaira dimension of $X$ is at most $1$, then $\Delta_{X}$ is nef if and only if $X$ has nef tangent bundle, with the exception of some properly elliptic surfaces which admit no section.  Surfaces with nef tangent bundle are classified by \cite{cp91}.
\end{itemize}

Note that any surface with nef diagonal must be minimal by Proposition \ref{nefimpliesnefcycles}.

\subsection{Kodaira dimension $-\infty$}

\begin{prop}
Let $X$ be a smooth surface of Kodaira dimension $-\infty$.  Then
\begin{enumerate}[(1)]
\item $\Delta_{X}$ is big if and only if $X = \mathbb{P}^{2}$.
\item $\Delta_{X}$ is nef if and only if $X$ has nef tangent bundle, or equivalently, if $X$ is either $\PP^2$, $\PP^1\times \PP^1$, or a projective bundle $\PP(\mathcal E)$ over an elliptic curve where $\mathcal{E}$ is either an unsplit vector bundle or (a twist of) a direct sum of two degree $0$ line bundles.
\end{enumerate}
\end{prop}

\begin{proof}
(1) Let $S$ be a smooth uniruled surface and let $g: S \to T$ be a map to a minimal model.  If $T$ is not $\mathbb{P}^{2}$, then $T$ (and hence also $S$) admits a surjective morphism to a curve.  If $T = \mathbb{P}^{2}$ and $g$ is not an isomorphism, then $g$ factors through the blow up of $\mathbb{P}^{2}$ at a point, which also admits a surjective morphism to a curve.  In either case Corollary \ref{bigmorcor} shows that the diagonal of $S$ is not big.

(2) We only need to consider minimal surfaces.  Using the classification, we see that any minimal ruled surface besides the ones listed carries a curve with negative self-intersection or maps to a curve of genus $\geq 2$.  By Proposition \ref{nefimpliesnefcycles} and Lemma \ref{neflowgenus} such surfaces can not have nef diagonal.
\end{proof}

\subsection{Kodaira dimension $0$}

\begin{prop}
The diagonal of a surface of Kodaira dimension $0$ is not big.
\end{prop}

\begin{proof}
By Lemma \ref{pushforwardpreservesbigdiag} it suffices to prove this for minimal surfaces.  Using the classification and Theorem \ref{qpg}, the only surface which could have a big diagonal would be an Enriques surface.  However, such surfaces always admit a map to $\mathbb{P}^{1}$ and thus can not have big diagonal by Lemma \ref{bigimpliesbigupsef}.
\end{proof}

We next turn to nefness of the diagonal.  Recall that any surface with nef diagonal must be minimal, and we argue case by case using classification.  Abelian surfaces and hyperelliptic surfaces both have nef tangent bundles, and thus nef diagonal.  For K3 surfaces, Theorem \ref{k3surfacenotnef} below verifies that the diagonal is never nef.

Finally, any Enriques surface admits an ample divisor $D$ with $D^2=2$ which defines a double ramified cover. Hence there is an involution $i:S\to S$ exchanging the two sheets. Then if $\Gamma_i$ is the graph, we have
$
\Delta_{S} \cdot \Gamma_i=-C^2<0.
$, and so $\Delta_{S}$ is not nef.

\subsubsection{K3 surfaces}

K3 surfaces are perhaps the most interesting example, and in this subsection we discuss them at some length.  We first discuss nefness, and we start with a couple low degree examples.

\begin{exmple} \label{deg2k3}
Let $S\to \PP^2$ be a degree 2 K3 surface. As for the Enriques surface, there is an involution $i:S\to S$, and intersecting $\Delta_{S}$ with the graph of the involution gives a negative number, so $\Delta_{S}$ is not nef.
\end{exmple}

\begin{exmple}\label{quarticcomputation}
Let $S$ be a surface in $\PP^3$, and let $W=\widetilde{S\times S}$ be the blow-up along the diagonal. Consider the divisor $H_1+H_2-E$, where $H_i$ is the pullback of the hyperplane section via the $i$-th projection. This divisor is base-point free, and defines a morphism
$$
\phi:W\to Gr(2,4).
$$
Geometrically, this is the morphism obtained by sending a pair of points on $S$ to the line they span; it is finite when $S$ contains no lines.

Now suppose that $S$ is a quartic K3 surface. Then
$$
(H_1+H_2-E)^2\pi^*\Delta_{S}=(H_1+H_2+\mathcal O(1))^2 \mathcal O(1)=(2H)^2-24=-8.
$$In particular, $\Delta_{S}$ has negative intersection with the images of the fibers of $\phi$.\end{exmple}

The previous example shows how knowledge of the nef cone of the blow-up $\widetilde{S\times S}$, or equivalently, $\Hilb^{2}(S)$, can be used to produce interesting subvarieties of $S\times S$ having negative intersection with $\Delta_{S}$. By the work of Bayer--Macr{\`{\i}, we can use similar arguments also for higher degrees.
\begin{thrm} \label{k3surfacenotnef}
Let $S$ be a K3 surface.  Then $\Delta_{S}$ is not nef.
\end{thrm}

We first prove a special case:

\begin{lem} \label{notnefk3}
Let $S$ be a K3 surface of Picard rank $1$ polarized by an ample divisor $H$ of degree $d \geq 4$.  Then $\Delta_{S}$ is not nef.
\end{lem}

\begin{proof}
We start by recalling the results of \cite{bm14} on the geometry of $\Hilb^{2}(S)$.  Suppose that $d/2$ is not a square.  It is clear that the fundamental solution to the Pell's equation $x^{2} - (d/2)y^{2} =1$ must have $x \geq \sqrt{d/2}$, so that the fundamental solution yields a ratio
\begin{equation*}
\frac{y}{x} = \sqrt{ \frac{2}{d} } \sqrt{ 1 - \frac{1}{x^{2}}} \geq \sqrt{ \frac{2}{d} } \sqrt{ \frac{d-2}{d}}.
\end{equation*}
Set $b_{d}  = \sqrt{\frac{d}{2} - 1} \leq \frac{d}{2} \cdot \frac{y}{x}$. 
Applying \cite[Proposition 13.1]{bm14}, we see that (whether or not $d/2$ is a square) the divisor class $H'-b_{d}B$ is movable on $\Hilb^{2}(S)$, where $H'$ is induced by the symmetric power of $H$ and $2B$ is the exceptional divisor for the Hilbert-Chow morphism.

We then apply Theorem \ref{hilb2thrm}.  The only verification necessary is:
\begin{equation*}
b_{d}^{2} = \frac{d}{2} - 1 > \frac{4d}{24}
\end{equation*}
which holds for $d$ in our range.
\end{proof}

In fact, the previous proof gives a little more: over the family of degree $d$ K3 surfaces, we have a class on the total space which restricts to be effective on a very general K3 surface and which has constant negative intersection against $\Delta_{S}$ for such surfaces.  Applying Proposition \ref{deformation}, we can take limits to deduce that for {\em every} K3 surface in the family, $\Delta_{S}$ has negative intersection against a pseudoeffective class.  This concludes the proof of Theorem \ref{k3surfacenotnef} in degree $\geq 4$, and we have already done the degree $2$ case in Example \ref{deg2k3}.

\subsubsection{Rigidity for K3 surfaces}

By again appealing to the results of \cite{bm14}, we can show rigidity under certain situations.

\begin{prop}
Let $S$ be a K3 surface of Picard number $1$ and degree $d$.  Suppose that the Pell's equation
\begin{equation*}
x^{2} - 2dy^{2} = 5
\end{equation*}
has no solutions.  Then the diagonal is strongly numerically rigid.
\end{prop}

For example, the theorem applies when the degree is divisible by $4$, or when the degree is less than $50$ except for degrees $2,10,22,38$.

\begin{proof}
By combining \cite[Lemma 13.3]{bm14} with the calculation in the proof of Lemma \ref{notnefk3}, we obtain the result from Theorem \ref{hilb2thrm}.
\end{proof}

Finally, we will prove that the diagonal of a very general K3 of degree $d$ is numerically rigid, using a deformation argument.

Standard results on K3 surfaces give the existence of a degree $d$ K3 surface $S_0$ which is also a quartic surface. It follows by the computation in Example \ref{quarticcomputation} that $\pi^*(\Delta_{S_0})$ is not pseudoeffective. Now take a family $\mathcal S\to T$ of polarized degree $d$ surfaces in a neighbourhood of $S_0$. Let $\pi:\widetilde{\S\times_{T} \S}\to \S\times_{T} \S$ be the blow-up of the diagonal. The induced family $\widetilde{\S\times_{T} \S}\to T$ is a smooth morphism. Consider the cycle class  $(\pi_t)^*(\Delta_{S_{t}})=(\pi^*\Delta_{\mathcal{S}/T})_t$. Since this is not pseudoeffective on the special fiber, $\pi^*(\Delta_{S_{t}})$ is not pseudoeffective for $t$ very general, by Proposition \ref{deformation}.  So applying again Proposition \ref{blowuprigidity}, we see that $\Delta$ is strongly numerically rigid on the very general K3 surface of degree $d$.

\begin{thrm}
Let $S$ be a very general polarized K3 surface. Then the diagonal is strongly numerically rigid.
\end{thrm}

\def\K{\mathcal K}
A posterori, this result is intuitive in light of the Torelli theorem, at least for subvarieties of $S\times S$ which are graphs of self-maps $f:S\to S$: if $\Gamma$ is such a graph and $[\Gamma]=[\Delta]$, then $f$ induces the identity on $H^2(S,\ZZ)$, and hence has to be the identity, and so $\Gamma=\Delta_S$.

\subsection{Kodaira dimension $1$}
Let $S$ be a surface of Kodaira dimension 1 and let $\pi:S\to C$ be the canonical map. By Corollary \ref{bigmorcor}, we have:

\begin{cor}
A surface of Kodaira dimension $1$ does not have big diagonal.
\end{cor}

We next show that the diagonal is not nef when $\pi$ admits a section.  As usual we may assume $S$ is minimal so that $K_{S}$ is proportional to some multiple of a general fiber of $\pi$.  Using Lemma \ref{neflowgenus}, we see that if the diagonal is nef then the base $C$ of the canonical map must have genus either $0$ or $1$.  If $T$ is a section of $\pi$, then by adjunction we see that $T^{2} < 0$, and so $\Delta_{S}$ is not nef (since as before $\Delta_{S} \cdot (\pi_{1}^{*}T \cdot \pi_{2}^{*} T)<0$).

\begin{exmple} \label{hyperellipticnefexample}
When $S\to C$ does not admit a section, it is possible for the diagonal to be nef. Indeed, let $E$ be an elliptic curve without complex multiplication and let $C$ be a hyperelliptic curve of genus $g$ which is very general in moduli.  The product $E \times C$ admits an involution $i$ which acts on $E$ as translation by a $2$-torsion point and on $C$ by the hyperelliptic involution. The quotient surface $S=(E\times C)/i$ is a properly elliptic surface of Kodaira dimension 1. The elliptic fibration $S\to C/i=\PP^1$ has a non-reduced fiber, and therefore can not admit a section. We claim that the diagonal of $S$ is nef. 

Let $S'=E\times C$ and let $\Gamma$ denote the graph of the involution $i$. By the projection formula it is enough to check that $\Delta_{S'}+\Gamma$ is nef on $S' \times S'$. Indeed, if $\pi:S'\to S$ is the quotient map, the map $\pi\times \pi$ is flat, and $(\pi\times \pi)^*\Delta_S=\Delta_{S'} +\Gamma$ is nef if and only if $\Delta_S$ is.

Let $f:S'\times S'\to C\times C$ denote the projection map $\pi_2\times \pi_2$.

\medskip

{\bf Claim:}
If an irreducible surface $T\subset E \times C \times E \times C$ is not nef, then it maps to a curve $D$ in $C \times C$ with negative self-intersection.  Furthermore it can only have negative intersection with surfaces contained in $f^{-1}(D)$.

\begin{proof}
It is clear that $T$ is nef if it maps to a point in $C \times C$.  We next prove nefness if $T$ maps dominantly onto $C \times C$.  Fix an irreducible surface $V$; we will show $T \cdot V \ge 0$.  It suffices to consider the case when $V$ is not a fiber of the map to $C \times C$.  In this situation we can deform $T$ using the abelian surface action so that it meets $V$ in a dimension $0$ subset.  Indeed, the set $Y\subset C\times C$ of points $y$ such that the fiber $T \cap f^{-1}(y)$ is $1$-dimensional is finite.  For a general translation in $E \times E$, this curve will meet $V \cap f^{-1}(y)$ in a finite set of points.  Next consider the open set $U = C \times C \backslash Y$.  Let $W \subset T$ be the subset lying over $U$.  We have a finite map from $E \times E \times W \to E \times E \times U$ given by $(a,b,w) \mapsto (a,b) \cdot w$.  In particular, the preimage of $V$ in $E \times E \times W$ will be a surface, and thus will meet a general fiber of $E \times E \times W \to E \times E$ properly.  Altogether we see a general translation of $T$ will meet $V$ properly.  Thus $T$ is nef.

Finally, suppose that $T$ maps to a curve $D$ and let $F=f^{-1}(D)$.  Suppose $V$ is a surface not contained in $F$.  Then $V \cdot T$ can be computed by restricting $V$ to $F$.  This restriction is effective, and hence nef (by the group action on $F$), showing that $T \cdot V \geq 0$.  If $V$ is also contained in $F$, then $V \cdot T$ in $X$ is the same as $f_*(v\cdot t)\cdot D$ on $C\times C$, where $V=i_*v$, $T=i_*t$. Note that $f_*(v\cdot t)$ will be a non-negative multiple of $D$ in $C\times C$. So if $D^2\ge 0$,  we again find that $T \cdot V \geq 0$. This completes the proof of the claim above.
\end{proof}

Note that if the diagonal $\Delta_S$ is not nef, there is a surface $T\subset S'\times S'$ with $T\cdot (\Gamma+\Delta_{S'})<0$. We must have either $\Gamma\cdot T<0$ or $\Delta_{S'} \cdot T<0$; replacing $T$ by $i(T)$, we may assume that the latter is the case. 

Arguing as above, we see $\Delta_{S'}$ and $T$ are both contained in the preimage $L$ of $\Delta_{C}$.   The intersection $\Gamma \cdot L$ is transversal; it consists of the points of $\Gamma$ over the 2-torsion points of $C$. In particular, the the restriction of $\Gamma$ to $L$ is numerically equivalent to 
\begin{equation*}
(\Delta_{C} \cdot \Gamma_{C})\Gamma_{E} = (2g+2) \Gamma_{E}
\end{equation*}
where $\Gamma_{E}$ is the pushforward of the graph of the involution from a fiber $E \times E$.  In contrast, $i^{*}i_{*}\Delta_{S}$ will be
\begin{equation*}
i^{*}i_{*}\Delta_{S} = (2-2g) \Delta_{E}
\end{equation*}
where $\Delta_{E}$ is the pushforward of the diagonal from a fiber $E \times E$.  Since $\Gamma_{E}$ and $\Delta_{E}$ are numerically proportional, we see that $\Delta_{S'} + \Gamma$ is nef when restricted to this threefold.  Hence its intersection with $T$ is non-negative, and so $\Delta_{S'}+\Gamma$ is nef overall.
\end{exmple}

\subsection{Surfaces of general type} \label{gentypesec}
\subsubsection{Surfaces with vanishing genus}

By Castelnuovo's formula, a surface of general type satisfying $p_{g} = 0$ also must satisfy $q=0$.  The minimal surfaces satisfying these conditions are categorized according to $K_{S}^{2}$, which is an integer satisfying $1 \leq K_{S}^{2} \leq 9$, and have Picard rank $10-K_{S}^{2}$.  It is interesting to look for examples where bigness holds or fails.

For such surfaces Theorem \ref{goodproducteff} shows:

\begin{cor}
Let $S$ be a smooth surface satisfying $p_{g}(S) = q(S) = 0$.  If $\Eff_{1}(S)$ is simplicial, then $\Delta_{S}$ is big if and only if every nef divisor is big.
\end{cor}

However, determining bigness can still be subtle.

{$\mathbf{K_{S}^{2} = 9}$:} The surfaces here are exactly the fake projective planes, and we saw in the introduction that $\Delta_{S}$ is both big and nef.

Suppose we blow-up a very general point to obtain a surface $Y$.  The results of \cite{steffens98} on Seshadri constants show that $Y$ carries a divisor which is nef and has self-intersection $0$, so the diagonal for $Y$ is neither nef or big.

{$\mathbf{K_{S}^{2} = 8}$:} Since these surfaces have Picard rank $2$, the pseudoeffective cone is automatically simplicial.  Thus we have an interesting trichotomy of behaviors:
\begin{itemize}
\item $\Eff^{1}(S) = \Nef^{1}(S)$.  Let $D_{1}$, $D_{2}$ be generators of the two rays of the pseudoeffective cone, and set $a = D_{1} \cdot D_{2}$.  Then
\begin{equation*}
\Delta_{S} = F_{1} + F_{2} + \frac{1}{a}\pi_{1}^{*}D_{1} \cdot \pi_{2}^{*}D_{2} + \frac{1}{a}\pi_{1}^{*}D_{2} \cdot \pi_{2}^{*}D_{1}
\end{equation*}
is nef.  However, $\Delta_{S}$ is not big since $S$ carries a non-zero nef class with self-intersection $0$.
\item If exactly one extremal ray of $\Eff^{1}(S)$ is nef, then $S$ carries both a curve of negative self-intersection and a nef class with vanishing self-intersection.  Thus $\Delta_{S}$ is neither big nor nef.
\item If no extremal rays of $\Eff^{1}(S)$ are nef, then $\Delta_{S}$ is big by Lemma \ref{goodproducteff} but is not nef.
\end{itemize}
There are a few known geometric constructions of such surfaces.  First, there are the surfaces constructed explicitly via ball quotients which are classified in \cite{dzambic14} and \cite{lsv15}.  Second, there are the surfaces admitting a finite \'etale cover which is a product of two curves.  Such surfaces are classified in \cite{bcg08} and are further subdivided into two types.  Write $S = (C_{1} \times C_{2})/G$ for some finite group $G$ acting on the product.  If no element of $G$ swaps the two factors, then $G$ acts on each factor separately.  This is known as the ``unmixed'' case.  There is also the ``mixed'' case, when $C_{1} \cong C_{2}$ and some elements of $G$ swap the two factors.

In the unmixed case, we have $C_{1}/G \cong C_{2}/G \cong \mathbb{P}^{1}$, and we obtain two maps $S \to \mathbb{P}^{1}$ such that the pullbacks of $\mathcal{O}(1)$ generate the pseudoeffective cone.  In particular $\Delta_{S}$ is nef but not big.  However, we do not know what happens in the other two situations.

\textbf{Burniat surfaces}:  These are certain surfaces constructed as Galois covers of weak del Pezzos (see for example \cite{alexeev13}).  By pulling back from the del Pezzo we obtain nef divisors with self-intersection $0$, so that $\Delta_{S}$ is not big.

\textbf{The Godeaux surface}: This surface is the quotient of the Fermat quintic by a $\mathbb{Z}/5$ action.  This surface admits a morphism to $\mathbb{P}^{1}$ (see for example the second-to-last paragraph on page 3 of \cite{gp02}), so $\Delta_{S}$ is not big.

\subsubsection{Surfaces with non-vanishing genus}

We next discuss several classes of surfaces of general type where we can apply our results.  These examples have $p_g>0$, and so $\Delta_{S}$ cannot be big. 

We note that by the computations of \cite{bc13}, if $H$ is a very ample divisor on a surface $S$ then (maintaining the notation of Section \ref{hilbschemesection}) $H^{[2]} - B'$ is nef. 

\begin{exmple}[Surfaces in $\PP^3$] \label{surfacesinP3}
Suppose that $S$ is a smooth degree $d$ hypersurface in $\mathbb{P}^{3}$.  Then $c_{2}(S) = d^{3} - 4d^{2} + 6d$.  Thus Theorem \ref{hilb2thrm} shows that the diagonal is strongly numerically rigid and is not nef as soon as $d \geq 4$.
\end{exmple}

\begin{exmple}[Double covers]
Suppose that $S$ is a double cover of $\mathbb{P}^{2}$ ramified over a smooth curve of even degree $d$.  Then $S$ is of general type once $d \geq 8$.  These surfaces have $c_{2}(S) = d^{2} - 3d+6$ and carry a very ample divisor of degree $d$.  Thus Theorem \ref{hilb2thrm} shows that the diagonal is strongly numerically rigid and is not nef as soon as $d \geq 8$.
\end{exmple}

\begin{exmple}[Horikawa surfaces]
Minimal surfaces of general type satisfying $q(S) = 0$ and $K_{S}^{2} = 2p_{g}(S) - 4$ are known as Horikawa surfaces and are studied by \cite{horikawa76}.  (These surfaces are the boundary case of Noether's inequality.)  The canonical map for such a surface defines a $2:1$ morphism onto a rational surface.  While $K_{S}$ is not very ample, it is big and basepoint free, which is enough to determine that $K_{S}^{[2]} - B'$ is movable on $\Hilb^{2}(S)$.  Using the equality $c_2(S) = 12 + 12p_{g}(S) - K_{S}^{2}$, we see that the diagonal for such surfaces is never nef by Theorem \ref{hilb2thrm}.
\end{exmple}

\section{Higher dimensional examples}

\subsection{Quadric hypersurfaces}\label{quadricsection}

An odd dimensional quadric is a fake projective space and thus will have big and nef diagonal as discussed in the introduction.  An even dimensional quadric will have diagonal that is nef but not big. Indeed, if $X$ is a quadric of dimension $2k$, then $X$ carries two disjoint linear spaces of dimension $k$. These linear spaces are nef (since $X$ is homogeneous), but not big, and hence $\Delta_{X}$ is not big.

\subsection{Nefness for hypersurfaces}

Example \ref{surfacesinP3} shows that the diagonal of a smooth hypersurface in $\mathbb{P}^{3}$ of degree at least $3$ is not nef.  The same is true in arbitrary dimension:

\begin{prop}
Let $X \subset \mathbb{P}^{n+1}$ be a smooth degree $d$ hypersurface.  If $d\geq 3$ then the diagonal for $X$ is not nef.
\end{prop}

\begin{proof}
The Euler characteristic of $X$ is
\begin{equation*}
c_n(X) = \frac{(1-d)^{n+2}-1}{d} + n + 2.
\end{equation*}
If $n$ is odd, then $\Delta_{X}^{2} = c_n(X) < 0$ and so $\Delta_{X}$ is not nef.  If $X$ has even dimension, then Corollary \ref{gonalitynefness} applied to the projection map from a general point outside $X$ shows that $\Delta_{X}$ is not nef. \end{proof}

\subsection{Bigness and rigidity for hypersurfaces}

Suppose that $X \subset \mathbb{P}^{n+1}$ is a smooth hypersurface of dimension $n$.  We will apply Lemma \ref{blowupcomputation} to show that the diagonal is not big for hypersurfaces of small dimension.  Note that homological non-bigness follows from Theorem \ref{qpg} once the degree is larger than $n+1$, so we will focus only on the Fano hypersurfaces.

Recall that Lemma \ref{blowupcomputation} requires us to find a class $\alpha \in N_{n}(X \times X)$ whose pullback under the blow-up of the diagonal $\phi: W \to X \times X$ has a special form.  We record the relevant information in the table below.  The first column records the kind of hypersurface.  The second column records the class $\alpha$ -- we will let $H_{1}$ and $H_{2}$ denote the pullback of the hyperplane class under the two projection maps.  

The third column verifies that Lemma \ref{blowupcomputation} applies to $\alpha$.  This is done by rewriting $\phi^{*}\alpha$ in terms of Schubert classes pulled back under the natural morphism $g: W \to G(2,n+2)$.  We will let $i: E \to W$ denote the inclusion of the exceptional divisor.  Identifying $E \cong \mathbb{P}(\Omega_{X})$ this divisor carries the class $\xi$ of the relative $\mathcal{O}(1)$ and  the pullback $H$ of the hyperplane class from the base of the projective bundle.  For convenience we let $h: E \to Gr(2,n+2)$ denote the restriction of $g$ to $E$. 

{\small 
\begin{center}
\begin{tabular}{c|c|c}
type & $\alpha$ & $\phi^{*}\alpha$ \\ \hline
cubic threefold & $H_{1}^{2}H_{2} + H_{1}H_{2}^{2} + \Delta_{X}$ & $g^{*}\sigma_{2,1} + i_{*}(h^{*}\sigma_{2} + 4H^{2})$ \\ \hline
quartic threefold & $2H_{1}^{2}H_{2} + 2H_{1}H_{2}^{2} + \Delta_{X}$ & $2g^{*}\sigma_{2,1} + i_{*}(h^{*}\sigma_{2} + 9H^{2})$ \\ \hline
cubic fourfold & $\begin{array}{c} H_{1}^{4} + H_{1}^{3}H_{2} + 3H_{1}^{2}H_{2}^{2} + \\ H_{1}H_{2}^{3} + H_{2}^{4} - \Delta_{X} \end{array}$ & $g^{*}\sigma_{4} + 2g^{*}\sigma_{2,2} + i_{*}(2h^{*}\sigma_{2,1} + 8H^{3})$  \\ \hline
quartic fourfold & $\begin{array}{c} H_{1}^{4} + H_{1}^{3}H_{2} + 7H_{1}^{2}H_{2}^{2} + \\ H_{1}H_{2}^{3} + H_{2}^{4}  - \Delta_{X} \end{array}$ & $g^{*}\sigma_{4} + 6g^{*}\sigma_{2,2} + i_{*}(3h^{*}\sigma_{2,1} + 24H^{3})$ \\ \hline
quintic fourfold & $\begin{array}{c} H_{1}^{4} + H_{1}^{3}H_{2} + 13H_{1}^{2}H_{2}^{2} + \\ H_{1}H_{2}^{3} + H_{2}^{4}  - \Delta_{X} \end{array}$ & $g^{*}\sigma_{4} + 12g^{*}\sigma_{2,2} + i_{*}(4h^{*}\sigma_{2,1} + 56H^{3})$ \\ \hline
cubic fivefold & $\begin{array}{c} H_{1}^{4}H_{2} + H_{1}^{3}H_{2}^{2} + \\ H_{1}^{2}H_{2}^{3} + H_{1}H_{2}^{4} + \Delta_{X} \end{array}$ & $g^{*}\sigma_{4,1} + i_{*}(h^{*}\sigma_{4} + 4 h^{*}\sigma_{2,2} + 12 H^{4})$ \\ \hline
quartic fivefold & $\begin{array}{c} 2H_{1}^{4}H_{2} + 11H_{1}^{3}H_{2}^{2} + \\ 11H_{1}^{2}H_{2}^{3} + 2H_{1}H_{2}^{4} + \Delta_{X}\end{array}$ & $\begin{array}{c} 2g^{*}\sigma_{4,1} + 9 H_{1}H_{2} g^{*}\sigma_{2,1} + \\ i_{*}(h^{*}\sigma_{4} + 9 h^{*}\sigma_{2,2} + 62 H^{4}) \end{array}$ \\ \hline
quintic fivefold & $\begin{array}{c} 3H_{1}^{4}H_{2} + 35H_{1}^{3}H_{2}^{2} + \\ 3H_{1}^{2}H_{2}^{3} + 3H_{1}H_{2}^{4} + \Delta_{X} \end{array}$ & $\begin{array}{c} 3g^{*}\sigma_{4,1} + 32 H_{1} g^{*}\sigma_{2,2} + \\ i_{*}(h^{*}\sigma_{4} + 16 h^{*}\sigma_{2,2} + 208 H^{4}) \end{array}$ \\ \hline
sextic fivefold & $\begin{array}{c} 4H_{1}^{4}H_{2} + 79H_{1}^{3}H_{2}^{2} + \\ 4H_{1}^{2}H_{2}^{3} + 4H_{1}H_{2}^{4} + \Delta_{X} \end{array}$ & $\begin{array}{c} 4g^{*}\sigma_{4,1} + 75 H_{1} g^{*}\sigma_{2,2} + \\ i_{*}(h^{*}\sigma_{4} + 25 h^{*}\sigma_{2,2} + 500 H^{4}) \end{array}$ \\ \hline
cubic sixfold & $\begin{array}{c} H_{1}^{6} + H_{1}^{5}H_{2} + 3H_{1}^{4}H_{2}^{2} + 3H_{1}^{3}H_{2}^{3} + \\ 3H_{1}^{2}H_{2}^{4} + H_{1}H_{2}^{5} + H_{2}^{6} - \Delta_{X} \end{array}$ & $\begin{array}{c} g^{*}\sigma_{6} + 2g^{*}\sigma_{4,2} + \\ i_{*}(2h^{*}\sigma_{4,1} + 8H h^{*}\sigma_{2,2} + 24H^{5}) \end{array}$
\end{tabular}
\end{center}
}
In all cases the class $\alpha$ satisfies $\alpha \cdot \Delta \leq 0$, and in all but the first the intersection is negative.  By Lemma \ref{blowupcomputation}, a smooth cubic threefold has non-big diagonal and in all other cases the diagonal is strongly numerically rigid.

It seems very likely that the same approach will work for all hypersurfaces of degree $\ge 3$ and dimension $\geq 3$.  Indeed, as the degree increases the coefficients are becoming more favorable.  We have verified strong numerical rigidity for several more cubic hypersurfaces but unfortunately the combinatorics become somewhat complicated.

\subsection{Grassmannians}
For a Grassmannian $X=Gr(k,n)$, the tangent bundle is globally generated, and hence the diagonal is nef. However, when $X$ is not a projective space, the diagonal is not big because the Schubert classes on $X$ are universally pseudoeffective but not big.

\subsection{Toric varieties}

\cite[Proposition 5.3]{fs09} shows that the only smooth toric varieties for which the pseudoeffective and nef cones coincide for all $k$ are products of projective spaces.  Thus products of projective spaces are the only toric varieties with $\Delta_{X}$ nef.  

Theorem \ref{goodproducteff} shows that for a toric variety $\Delta_{X}$ is big if and only if every nef class on $X$ is big.  Although it seems reasonable to expect that the only toric varieties with big diagonal are projective spaces, this turns out to be false.

\begin{exmple}
\cite{fs09} constructs an example of a toric threefold which does not admit a morphism to a variety of dimension $0 < d <3$.  In particular, this implies that every nef divisor is big.  Dually, every nef curve is big.  Thus $X$ has big diagonal.
\end{exmple}

Note that, aside from projective space, the diagonal can not be big for smooth projective toric varieties of Picard rank $\leq 3$ (or for threefolds of Picard rank $\leq 4$), since these varieties admit a non-constant morphism to a lower-dimensional variety \cite{rt15}.

\subsection{Varieties with nef tangent bundle}

Varieties with nef tangent bundle are expected to have a rich geometric structure.  A conjecture of Campana--Peternell \cite{cp91} predicts that such varieties are (up to \'etale cover) flat bundles with rational homogeneous fibers over their albanese variety.  To prove this result, it suffices by a result of \cite{dps94} to show that a Fano variety with nef tangent bundle is rational homogeneous.  This conjecture has been verified up to dimension 5 and in several related circumstances (see \cite{mok88}, \cite{hwang01}, \cite{mok02}, \cite{watanabe14}, \cite{kanemitsu15}, \cite{watanabe15}).

A variety with nef tangent bundle has nef diagonal.  However, the diagonal can only be big if the albanese map is trivial.

\begin{exmple}
A variety with nef tangent bundle need not carry an action of an algebraic group -- this can only be expected after taking an \'etale cover.  An explicit example is discussed in \cite[Example 3.3]{dps94}.

However, the diagonal can still be universally pseudoeffective in this situation.  Indeed, suppose we have an \'etale cover $f: Y \to X$ where $Y$ admits a transitive action of an algebraic group.  The induced flat finite map $f^{\times 2}: Y \times Y \to X \times X$ has that $f^{\times 2}_{*} \Delta_{Y} = d \Delta_{X}$.  Of course the diagonal of $Y$ is universally pseudoeffective and since universal pseudoeffectiveness is preserved by flat pushforward for smooth varieties by \cite[Theorem 1.8]{fl14} we deduce that $\Delta_{X}$ is also universally pseudoeffective.
\end{exmple}

\section{Threefolds}

In this section we discuss a couple classification results for threefolds.

\subsection{Fano threefolds}

The Fano threefolds which are fake projective spaces are are $\mathbb{P}^{3}$, the quadric, the Del Pezzo quintic threefolds $V_{5}$, and the Fano threefolds $V_{22}$; see \cite{ls86}.  As discussed above these will have big and nef diagonal.

Any Fano threefold with nef diagonal must be primitive and the contractions corresponding to the extremal rays of $\Eff_{1}(X)$ can not be birational.  Another obstruction to nefness is the fact that $\Delta_{X}^{2}$ is the topological Euler characteristic of $X$.  In addition to the fake projective spaces discussed earlier, the classification of \cite{AGVFANO} and \cite{Ott05} leaves $7$ possibilities:
\begin{itemize}
\item \textbf{Picard rank 1:} $V_{18}$, the intersection of two quadrics in $\mathbb{P}^{5}$.
\item \textbf{Picard rank 2:} a double cover of $\mathbb{P}^{2} \times \mathbb{P}^{1}$ with branch locus $(2,2)$, a divisor on $\mathbb{P}^{2} \times \mathbb{P}^{2}$ of type $(1,2)$ or $(1,1)$, $\mathbb{P}^{1} \times \mathbb{P}^{2}$.
\item \textbf{Picard rank 3:}  $\mathbb{P}^{1} \times \mathbb{P}^{1} \times \mathbb{P}^{1}$.
\end{itemize}
The threefolds on this list with nef tangent bundle are classified by \cite{cp91} and will automatically have nef diagonal: these are the products of projective spaces and the $(1,1)$ divisor in $\mathbb{P}^{2} \times \mathbb{P}^{2}$.   The divisor of type $(1,2)$ on $\mathbb{P}^{2} \times \mathbb{P}^{2}$ will also have nef diagonal, which we can see as follows.  Let $\{ D_{1}, D_{2} \}$ be the basis of $N_{2}(X)$ consisting of extremal nef divisors and $\{ C_{1}, C_{2} \}$ denote the basis of $N_{1}(X)$ consisting of extremal nef curves.  Since $h^{1,2}(X)=0$ we can write the diagonal very explicitly using products of these basis elements and it is easy to see that it is a non-negative sum of nef classes.  We are unsure of what happens in the remaining $3$ cases.

Any Fano threefold with big diagonal must have that the contractions corresponding to the extremal rays of $\Eff_{1}(X)$ are birational.  We focus on the primitive Fano threefolds; going through the classification of \cite{Ott05}, we see that any such threefold must have Picard rank $1$.  Unfortunately it seems subtle to determine which of these threefolds actually have big diagonal.

\subsection{Threefolds with nef and big diagonal}

\begin{prop}
Let $X$ be a smooth minimal threefold of Kodaira dimension $\geq 0$.  Then $X$ does not have homologically big diagonal.
\end{prop}

\begin{proof}
If $\Delta_{X}$ is homologically big, then $H^{k,0}(X)=0$ for all $k>0$. In particular, $\chi(\O_X)=1$. By Riemann--Roch, $\chi(\O_X)=\frac{1}{24}c_1c_2$, so we find that $c_1c_2=24$. In particular, $c_1(X)\neq 0$. If $0<\kappa(X)<3$, then $X$ admits a map to a lower-dimensional variety, contradicting bigness of $\Delta_{X}$. If $X$ has general type, we have by the Miayoka--Yau inequality, 
$$
0>-K^3=c_1^3\ge \frac83 c_1c_2=64
$$
This is a contradiction. 
\end{proof}

\begin{cor}
Let $X$ be a smooth threefold such that $\Delta_{X}$ is nef and homologically big.  Then $X$ is a fake projective space: $\mathbb{P}^{3}$, the quadric, the Del Pezzo quintic threefolds $V_{5}$, the Fano threefolds $V_{22}$.
\end{cor}

\begin{proof}
Suppose first that $\kappa(X) \geq 0$.  Since $\Delta_{X}$ is nef, $X$ must be a minimal threefold.  We conclude by the previous proposition that the diagonal can not be big.

Thus we know that $X$ is uniruled.  Furthermore $N^1(X)=\mathbb R$, and so $X$ is a Fano threefold of Picard number 1. 
Note also that $\chi(X)=\Delta_{X}^2\ge 1$, so in particular $h^{2,1}$ is at most 1. Going through the classification of such Fano threefolds \cite{AGVFANO} reveals that the only possibilities are the fake projective 3-spaces.
\end{proof}

Recall that when $X$ is a threefold not of general type then numerical and homological equivalence coincide on $X \times X$.  Thus the corollary also classifies the threefolds not of general type which have nef and big diagonal.

\section{Cohomological decomposition of the diagonal and positivity} \label{blochconjsec}

This section discusses the relationship of Theorem \ref{bignessthrm} with the decomposition of the diagonal in cohomology on a surface. The following result is surely well-known to experts, but we include a proof for the convenience of the reader.

%
%\begin{prop} \label{numbloch}
%Let $S$ be a smooth projective surface. Then the following are equivalent:
%\begin{enumerate}
%\item $p_{g}(S) = q(S) = 0$
%\item $\Delta_S$ is homologous to a sum of cycles contracted by the projection maps
%\item  $\Delta_{S}$ is homologous to  a sum of products of pullbacks of classes from the two projections
%
%\end{enumerate}
%\end{prop}
%
%\begin{proof}
%
%(1)$\implies$(3). The hypothesis guarantees that all of $H^{1,1}(S)$ is algebraic.  Applying the K\"unneth formula we see that all of $H^{2,2}(X \times X)$ is algebraic and that
%\begin{equation*}
%N_{2}(S \times S) = \oplus_{i+j = 2} \pi_{1}^{*}N_{i}(S) \cdot \pi_{2}^{*}N_{j}(S).
%\end{equation*}which gives (3). 
%
%For (2)$\implies$(1), note that the assumption implies that $[\Delta]$ is homologous to a cycle $Z$ which is supported in a divisor $D\times S$. Arguing as in Theorem 3.13 of \cite{voisin2014chow}, we find that $h^{1,0}(S)=h^{2,0}(S)=0$. The implication (3)$\implies$(2) is trivial. 
%\end{proof}

\begin{prop} \label{numbloch}
Let $S$ be a smooth projective surface.  Then:
\begin{enumerate}
\item $\Delta_{S}$ is homologous to a sum of cycles contracted by the projection maps if and only if $p_{g}(S) = 0$.
\item $\Delta_{S}$ is homologous to  a sum of products of pullbacks of classes from the two projections if and only if $p_{g}(S) = q(S) = 0$.
\end{enumerate}
\end{prop}
\begin{proof}
We first prove (2).  Note that $\Delta$ is homologous to a sum $Z+Z'$ where $Z$ is a cycle supported in a divisor $D\times S$ and $Z'$ is supported on a fiber $S\times s$ of the projection map. By a Bloch--Srinivas type argument (as in \cite[Theorem 10.17]{voisinH}), we find that $p_g(S)=q(S)=0$. For the reverse implication, the K\"unneth formula implies that all of $H^{2,2}(S \times S)$ is algebraic and generated by products of pullbacks of divisors from the two projections. 

Now we prove (1).  The forward implication follows again from the argument of \cite[Theorem 10.17]{voisinH}.  Conversely, the arguments of \cite{bkl76} using the classification theory of surfaces show that for surfaces with $\kappa(S) < 2$ and $p_{g}(S) = 0$ there is a curve $C \subset S$ such that $\CH_{0}(C) \to \CH_{0}(S)$ is surjective.  Using the decomposition of the diagonal as in \cite{bs83}, we see that the reverse implication holds except possibly for surfaces of general type.  But a surface of general type satisfying $p_{g}(S) = 0$ also satisfies $q(S) = 0$ by Castelnuovo's theorem.  Thus we are reduced to (2).
\end{proof}

%\begin{rmk}
%While the work of \cite{murre90} uses the language of motives, the construction is in fact quite geometric, and one can give a ``direct'' proof of the previous theorem in the following way.  Consider again the forward implication in (2): we must show that $q(S) = 0$.  Consider the albanese $\mathrm{alb}^{\times 2}: X \times X \to A \times A$.  If the albanese map is not trivial, one obtains a contradiction to the hypothesis by considering the intersection of $\mathrm{alb}^{\times 2}_{*} \Delta_{X}$ against the (pullback of) the Poincare bundle and comparing with the intersections against $\pi_{1}^{*}N_{1}(X) \cdot \pi_{2}^{*}N_{1}(X)$.
%\end{rmk}

It is natural to ask whether one can obtain a tighter link between positivity and decompositions of $\Delta_{X}$ than Theorem \ref{qpg}.  Following an idea of \cite{djv13}, we will prove such a statement for surfaces by perturbing the diagonal by an external product of ample divisors.

\begin{prop}
Let $S$ be a smooth projective surface.  Then $p_{g}(S) = 0$ if and only if there is an ample divisor $H$ such that $\Delta_{S} + \pi_{1}^{*}H \cdot \pi_{2}^{*}H$ is big.
\end{prop}

\begin{proof}
We first prove the forward implication.  By Proposition \ref{numbloch}, we have an equality of numerical classes
\begin{equation*}
\Delta_{S} = a_{0}F_{1} + b_{0}F_{2} + \sum_{i=1}^{r_{1}} a_{i}E_{i} + \sum_{j=1}^{r_{2}} b_{j}E_{j}'
\end{equation*}
where $a_{i},b_{j} \in \mathbb{Q}$, each $E_{i}$ is an irreducible surface contracted to a curve by $\pi_{1}$, and each $E_{j}'$ is an irreducible surface contracted to a curve by $\pi_{2}$.  Note that
\begin{equation*}
a_{0} + \sum_{i=1}^{r_{1}} a_{i}E_{i} \cdot F_{2} = 1 \qquad \qquad b_{0} + \sum_{j=1}^{r_{2}} b_{j}E_{j}' \cdot F_{1} = 1
\end{equation*}
For each $i$, let $C_{i}$ denote the normalization of the image $\pi_{1}(E_{i})$ and let $D_{i}$ denote $C_{i} \times S$.  For notational convenience, we will omit the normalization and write $C_{i}$ and $D_{i}$ as if they were subvarieties of $S$ and $S \times S$.

Fix a small $\epsilon > 0$ satisfying $\epsilon < 1/r_{1}$ and $\epsilon < 1/r_{2}$.  Set $c_{i} = -a_{i}E_{i} \cdot F_{2}$ so that
\begin{equation*}
a_{i}E_{i} + \left(c_{i}+ \epsilon \right) F_{1}
\end{equation*}
 is $\pi_{2}$-relatively ample as a divisor on $D_{i}$.  Thus, for some sufficiently large ample $H_{i}$ on $S$, we have that
\begin{equation*}
a_{i}E_{i} +\left( c_{i} +  \epsilon \right) F_{1} + \pi_{2}|_{D_{i}}^{*}H_{i}
\end{equation*}
is an effective class on $D_{i}$.  Pushing forward to $S \times S$ and adding up as $i$ varies, we see that
\begin{equation*}
\left( r_{1}\epsilon + a_{0} - 1 \right) F_{1} + \sum_{i=1}^{r_{1}} \left( a_{i}E_{i} + \pi_{1}^{*}C_{i} \cdot \pi_{2}^{*}H_{i} \right)
\end{equation*}
is an effective class on $S \times S$.  Arguing symmetrically, with analogous notation,
\begin{equation*}
\left( r_{2}\epsilon + b_{0} - 1 \right) F_{2} + \sum_{j=1}^{r_{2}} \left( b_{j}E_{j}' + \pi_{1}^{*}H_{j}' \cdot \pi_{2}^{*}C_{j}' \right)
\end{equation*}
is an effective class.  Of course, we can replace the $C_{i}, H_{i}, C_{j}', H_{j}'$ by larger ample divisors without affecting the effectiveness of this class.  All told, there is an effective surface class $Q$ and a positive sum of external products of ample divisors $N$ such that
\begin{equation*}
\Delta_{S} + N = (1-r_{1}\epsilon)F_{1} + (1-r_{2}\epsilon)F_{2} + Q.
\end{equation*}
Adding on a further external product of amples to both sides we can ensure that the right hand side is big: if $A$ is ample on $X$, a class of the form $c_{1}F_{1} + c_{2}F_{2} + c_{3}\pi_{1}^{*}A \cdot \pi_{2}^{*}A$ with positive coefficients will dominate a small multiple of the big class $(\pi_{1}^{*}A + \pi_{2}^{*}A)^{2}$.  Finally, any positive sum of external products of amples is dominated by a single external product of amples, finishing the proof of the implication.

Conversely, suppose $h^{2,0}(X) > 0$ and let $\alpha \in H^{2,0}(X)$ be non-zero.  Let $\beta$ be the nef class constructed as in Theorem \ref{qpg} satisfying $\beta \cdot \Delta_{S} = 0$.  This $\beta$ also satisfies $\beta \cdot \pi_{1}^{*}H \cdot \pi_{2}^{*}H=0$.  Since $\beta$ is nef, there can be no big class of the desired form (again appealing to the equality of homological and numerical equivalence).
\end{proof}
%\begin{align*}
%\beta \cdot \pi_{1}^{*}H \cdot \pi_{2}^{*}H  & = -2 \int_{H} \alpha \int_{H} \bar{\alpha}=0.
%\end{align*}

\section{Questions}\label{questions}

We finish with a list of questions raised by our work.

\begin{ques}
Are $\mathbb{P}^{2}$ and fake projective planes the only smooth surfaces with big diagonal?
\end{ques}

\begin{ques}
Suppose that $S$ is a smooth surface $S$ of general type with $q(S) = 0$ and $p_g(S) > 0$.  Is the diagonal for $S$ numerically rigid?
\end{ques}

It is natural to ask whether some of the results for surfaces generalize for higher dimensions:

\begin{ques}
Does a smooth projective variety with big and nef diagonal have the same rational cohomology as projective space?
\end{ques}

\begin{ques}
Are the only smooth projective varieties with big diagonal either uniruled or of general type?
\end{ques}

\begin{ques}
Is there a threefold of general type with nef diagonal?
\end{ques}

\begin{ques}
Are there any topological restrictions on smooth varieties with nef diagonal aside from $c_{n}(X) \geq 0$?  For example, does a threefold with nef diagonal satisfy $\chi(\mathcal{O}_{X}) \geq 0$?
\end{ques}

%\nocite{*}
\bibliographystyle{alpha}
{
\bibliography{posdiag}
}
%
%{Department of Mathematics, Boston College, Chestnut Hill, MA  02467, USA}
%
%{{\it Email:} \verb"lehmannb@bc.edu"}
%
%\medskip
%
%%\small{
%{Department of Mathematics, University of Oslo, Box 1053, Blindern, 0316 Oslo, Norway}
%%}
%
%{{\it Email:} \verb"johnco@math.uio.no"}
%

\end{document}